\documentclass[12pt]{article}
\usepackage{graphicx}
\usepackage{subfigure}
\usepackage{amssymb,amsmath,amsthm,epsfig}
\usepackage{latexsym, enumerate}
\usepackage{eepic}
\usepackage{epic}
\usepackage{graphicx}
\usepackage{color}
\usepackage{ifpdf}
\usepackage{booktabs}

\usepackage{dsfont}
\usepackage{multirow}
\usepackage{bm}

\theoremstyle{plain}
\newtheorem{lem}{Lemma}[section]
\newtheorem{thm}[lem]{Theorem}

\theoremstyle{definition}

\theoremstyle{remark}
\newtheorem{rem}{Remark}[section]
\newtheorem{exm}{Example}[section]

\begin{document}
\title{ \large\bf On uniqueness and nonuniqueness for potential reconstruction in quantum fields from one measurement}

\author{
Guang-Hui Zheng\thanks{Corresponding author. College of Mathematics and Econometrics, Hunan University, Changsha 410082, Hunan Province, China. Email: zhenggh2012@hnu.edu.cn}
\and
Zhi-Qiang Miao\thanks{College of Mathematics and Econometrics, Hunan University, Changsha 410082, Hunan Province, China. Email: zhiqiang\_miao@hnu.edu.cn}
}

\date{}
\maketitle

\begin{abstract}
  This paper studies uniqueness and nonuniqueness for potential reconstruction from one boundary measurement in quantum fields, associated with the steady state Schr\"{o}dinger equation. A uniqueness theorem of the inverse problem is established. In the meanwhile, a nonuniqueness theorem is also given when different potential and shape are considered. Finally, Tikhonov regularization method is applied to solve the reconstruction problem, and some numerical examples are presented to confirm the theoretical results and the effectiveness of the proposed method.
\end{abstract}\smallskip

\smallskip
{\bf keywords}: Potential reconstruction, Schr\"{o}dinger equation, Dirichlet-to-Neumann map, modified Bessel function.

\section{Introduction}
Let $\Omega$ be an open bounded domain in $\mathbb{R}^2$ containing origin possibly with multiple components with a
smooth boundary $\partial\Omega$, and $\nu$ be the outward unit normal vector to $\partial\Omega$.  Then we consider the steady state Schr\"{o}dinger equation \cite{Cohen-Tannoudji-1991} as follows
\begin{eqnarray}\label{1.1}
\left(-\frac{\hbar^2}{2m}\Delta+U(x)\right)\psi=E\psi,\ \ \ \ \ \text{in}\ \Omega,
\end{eqnarray}
with Dirichlet boundary condition
\begin{equation}\label{1.2}
\psi=f,\ \ \ \ \ \ \text{on}\ \partial\Omega,
\end{equation}
where $\hbar$, $m$ denote the reduced Planck's constant and the mass of particles respectively, $U(x)$ is the potential, $E$ is the energy value, and the solution $\psi(x)$ is called de Broglie's matter wave.

In this paper, we focus on solving the following inverse problem,

\emph{Inverse problem}: For arbitrary fixed energy value $E>0$, recover the potential $U(x)$ from one boundary measurement $\frac{\partial\psi}{\partial\nu}\bigg|_{\partial\Omega}$.

The above inverse problem is also related closely with the classical Calder\'{o}n problem, which was formulated by Calder\'{o}n in \cite{Calderon1980}. In 1987, Sylvester and Uhlmann proved the uniqueness with many boundary measurements in $\mathbb{R}^3$ for $U\in C(\overline{\Omega})$. Nowadays there are many generalizations concerning this result. For example, the case of partial boundary measurements $\left(\psi|_{\Gamma},\ \frac{\partial\psi}{\partial\nu}\bigg|_{\Gamma}\right)$, here $\Gamma$ is an arbitrary fixed open set of $\subseteq\partial\Omega$. One can refer to \cite{A.L.BUKHGEIM and G.UHLMANN-2002, Isakov2007, Imanuvilov2010, Imanuvilov2013} and a survey paper \cite{Kenig2014}. For the uniqueness results of Calder\'{o}n problem with single boundary measurement, to our knowledge, the first result is given by Isakov in 1989 \cite{Iskov1989}. Recently, Alberti and Santacesaria established the uniqueness, stability estimates and reconstruction algorithm for determining the potential in (\ref{1.1})-(\ref{1.2}) from a finite number of boundary measurements \cite{Alberti2018}. The studies of nonuniqueness are closely connected to the researches about invisibility \cite{Greenleaf2003, Greenleaf2003-1, Greenleaf2007, Greenleaf2009} and virtual reshaping \cite{Liu2009}. In particular,  Greenleaf, Lassas and Uhlmann construct some counterexamples to uniqueness in Calder\'{o}n problem by transformation optics \cite{Greenleaf2003, Greenleaf2003-1}. Moreover, in \cite{Liu2009}, Liu also used transformation optics to reshape an obstacle in acoustic and electromagnetic scattering.

More specifically, we consider the determination of piecewise constant potential in the unit disc. based on the analytic formula of solution, and notice the monotonicity of modified Bessel function, we prove that only one boundary measurement can recover the potential uniquely. Furthermore, by choosing appropriate potential and radius of the core for different core-shell structure, the boundary data $\frac{\partial\psi}{\partial\nu}\bigg|_{\partial\Omega}$ can be concordant. In other words, one boundary data are not able to determine the shape of core and potential simultaneously. Finally, we applied Tikhonov regularization method to recover the potential, and some numerical examples are given to verify the corresponding theoretical results.

This paper is organized as follows :
In section 2, the solution formula of the forward problem and the associated Dirichlet to Neumann map are given. The uniqueness theorem of potential reconstruction from one boundary measurement is established in section 3. The nonuniqueness result is obtained in section 4. In section 5, we present some numerical results to confirm the theoretical analysis, and a conclusion is given in Section 6. 
%
%

\section{Solution formula and Dirichlet to Neumann map}
In this section, based on the polar coordinate transformation, we deduce the exact solution formula of (\ref{1.1})-(\ref{1.2}) in core-shell structure, and define the Dirichlet to Neumann map (DN map).

Multiplying each side by $-\frac{2m}{\hbar^2}$ in (\ref{1.1}), we have
\begin{equation}\label{2.1}
  \Delta\psi(x)-\widetilde{U}(x)\psi(x)=-\widetilde{E}\psi(x),
\end{equation}
where $\widetilde{U}(x)=\frac{2m}{\hbar^2}U(x)$, $\widetilde{E}=\frac{2m}{\hbar^2}E$.
Let $\Omega$ be an annulus of radius $r_1$ and $1$ (core-shell structure), and the potential $U(x)$ be a piecewise constant function, i.e.
\begin{eqnarray}\label{2.3}
\widetilde{U}(x)=
\begin{cases}
\widetilde{U}_1,\ \ \ \ |x|<r_1,\\
\widetilde{U}_2,\ \ \ \ r_1<|x|<1.
\end{cases}
\end{eqnarray}
For simple, we set $\widetilde{U}_1=\widetilde{E}+\sigma_1^{-1}$, $\sigma_1>0$, and $\widetilde{U}_2=\widetilde{E}+1$. Then, under the polar coordinates, (\ref{1.1}) becomes
\begin{eqnarray}\label{2.4}
\frac{1}{r}\frac{\partial}{\partial r}\left(r\frac{\partial\psi}{\partial r}\right)-\left(\sigma_1^{-1}\chi_{\{r<r_1\}}+1\chi_{\{r_1<r<1\}}\right)\psi=0,\ \ \ \ \ r\in(0,1).
\end{eqnarray}
Assume that $\psi|_{r=0}$ is bounded. The corresponding DN map can be written as
\begin{equation}\label{2.5}
\Lambda_{\sigma_1,r_1}:\ \mathbb{C}\ni\psi|_{r=1}\ \mapsto\ \frac{\partial\psi}{\partial r}\bigg|_{r=1}\in\mathbb{C}.
\end{equation}
Furthermore, setting $\psi|_{r=1}=f$ and $\frac{\partial\psi}{\partial r}\big|_{r=1}=g$, the DN map (\ref{2.5}) can be represented
by solving problem (\ref{2.4}) with Dirichlet boundary condition $\psi|_{r=1}=f$:
\begin{eqnarray}\label{2.6}
\begin{cases}
r^{-1}\frac{\partial}{\partial r}\left(r\frac{\partial\psi}{\partial r}\right)-\sigma_1^{-1}\psi=0,\ \ \ \ \ r\in(0,r_1),\\
r^{-1}\frac{\partial}{\partial r}\left(r\frac{\partial\psi}{\partial r}\right)-\psi=0,\ \ \ \ \ r\in(r_1,1),\\
\psi|_{r=r_1}^+=\psi|_{r=r_1}^-,\\
\frac{\partial\psi}{\partial r}\big|_{r=r_1}^+=\sigma_1\frac{\partial\psi}{\partial r}\big|_{r=r_1}^-,\\
\psi|_{r=1}=f,\\
\psi|_{r=0}\ \ \text{is bounded}.
\end{cases}
\end{eqnarray}
where $\cdot|_{r=r_1}^+$ means the limit to the outside of $\{r|r=r_1\}$ and $\cdot|_{r=r_1}^-$ means the limit to the inside of $\{r|r=r_1\}$.

By the boundedness of $\psi(0)$, we suppose that the matter wave $\psi(r)$ has the form
\begin{eqnarray}\label{2.7}
\psi(r)=
\begin{cases}
a_0I_0\left(\frac{r}{\sqrt{\sigma_1}}\right),\ \ \ \ \ r\in(0,r_1),\\
a_1I_0\left(r\right)+b_1K_0\left(r\right),\ \ \ \ \ r\in(r_1,1),
\end{cases}
\end{eqnarray}
where $I_n\left(r\right)$ and $K_n\left(r\right)$ $(n\in\mathbb{N})$ denote the $n$-th order modified Bessel functions of the first and the second kind, respectively.
$a_0$, $a_1$, $b_1$ are unknown coefficients.

From the transmission conditions on the interface $\{r|r=r_1\}$ and boundary value condition on $\{r|r=1\}$, we have that
\begin{eqnarray}\label{2.8}
\begin{cases}
a_0I_0\left(\frac{r_1}{\sqrt{\sigma_1}}\right)=a_1I_0\left(r_1\right)+b_1K_0\left(r_1\right),\\
a_0\sqrt{\sigma_1}I_1\left(\frac{r_1}{\sqrt{\sigma_1}}\right)=a_1I_1\left(r_1\right)-b_1K_1\left(r_1\right),\\
a_1I_0\left(1\right)+b_1K_0\left(1\right)=f,
\end{cases}
\end{eqnarray}
here we used the differential formula $I_0'\left(r\right)=I_1\left(r\right)$, $K_0'\left(r\right)=-K_1\left(r\right)$.

By solving (\ref{2.8}), we obtain
\begin{eqnarray}\label{2.9}
\begin{cases}
a_0=\frac{\left(\rho(r_1,\sigma_1)K_0(r_1)-I_0(r_1)\right)f}{\left(\rho(r_1,\sigma_1)K_0(1)-I_0(1)\right)I_0
\left(\frac{r_1}{\sqrt{\sigma_1}}\right)},\\
a_1=-\frac{f}{\rho(r_1,\sigma_1)K_0(1)-I_0(1)},\\
b_1=\frac{\rho(r_1,\sigma_1)f}{\rho(r_1,\sigma_1)K_0(1)-I_0(1)},
\end{cases}
\end{eqnarray}
where
\begin{eqnarray}\label{2.10}
\rho(r_1,\sigma_1)=\frac{\sqrt{\sigma_1}I_1\left(\frac{r_1}{\sqrt{\sigma_1}}\right)I_0\left(r_1\right)-I_0\left(\frac{r_1}
{\sqrt{\sigma_1}}\right)I_1\left(r_1\right)}{\sqrt{\sigma_1}I_1\left(\frac{r_1}{\sqrt{\sigma_1}}\right)K_0\left(r_1\right)+I_0\left(\frac{r_1}
{\sqrt{\sigma_1}}\right)K_1\left(r_1\right)}.
\end{eqnarray}
Hence, by substituting the coefficient formula (\ref{2.9}) into (\ref{2.7}), we get the solution of problem (\ref{2.6}). Finally, the DN map can be expressed precisely as follows,
\begin{equation}\label{2.11}
\Lambda_{\sigma_1,r_1}(f)=-\frac{\rho(r_1,\sigma_1)K_1(1)+I_1(1)}{\rho(r_1,\sigma_1)K_0(1)-I_0(1)}f.
\end{equation}
Clearly, $\Lambda_{\sigma_1,r_1}:\ \mathbb{C}\rightarrow\mathbb{C}$ is a multiplier operator, and its operator norm is defined by
\begin{equation}\label{2.111}
\|\Lambda_{\sigma_1,r_1}\|=\sup_{f\in\mathbb{C}}\frac{|\Lambda_{\sigma_1,r_1}(f)|}{|f|}
\end{equation}
From (\ref{2.11}), it implies
\begin{equation}\label{2.112}
\|\Lambda_{\sigma_1,r_1}\|=\left|\frac{\rho(r_1,\sigma_1)K_1(1)+I_1(1)}{\rho(r_1,\sigma_1)K_0(1)-I_0(1)}\right|.
\end{equation}

Next, we define the following DN map for Schr\"{o}dinger equation in a disk:
\begin{equation}\label{2.12}
\Lambda(f)=\frac{\partial\Psi}{\partial r}\bigg|_{r=1},
\end{equation}
where $\Psi$ is the solution to
\begin{eqnarray}\label{2.13}
\begin{cases}
r^{-1}\frac{\partial}{\partial r}\left(r\frac{\partial\Psi}{\partial r}\right)-\Psi=0,\ \ \ \ \ r\in(0,1),\\
\Psi|_{r=1}=f,\\
\Psi|_{r=0}\ \ \text{is bounded}.
\end{cases}
\end{eqnarray}
Similarly, the DN map (\ref{2.12}) can be also represented by
\begin{equation}\label{2.14}
\Lambda(f)=\frac{I_1(1)}{I_0(1)}f.
\end{equation}

\section{Uniqueness for the potential reconstruction in core-shell structure}
The main purpose of this section is to establish the uniqueness theorem for the potential reconstruction problem from one boundary measurement, i.e., determining the piecewise constant potential in the core-shell structure by $\frac{\partial\psi}{\partial r}\bigg|_{r=1}$. First, we give the asymptotic property of DN map respect to the radius of core $r_1$ and potential coefficient $\sigma_1$.
\begin{thm}\label{th1}
Let $\Lambda_{\sigma_1,r_1}$, $\Lambda$ are defined by (\ref{2.5}) and (\ref{2.12}) respectively. Then,\\
(1)\ for any fixed $r_1\in(0,1)$, we have that
\begin{equation}\label{2.15}
\|\Lambda_{\sigma_1,r_1}-\Lambda\|\rightarrow0,\ (as\ \sigma_1\rightarrow1);
\end{equation}
(2)\ for any fixed $\sigma_1>0$, we have that
\begin{equation}\label{2.151}
\|\Lambda_{\sigma_1,r_1}-\Lambda\|\rightarrow0,\ (as\ r_1\rightarrow0);
\end{equation}
\end{thm}

\begin{proof}

(1)\ It is a straightforward consequence of the definitions of operator norm for the DN map $\Lambda_{\sigma_1,r_1}$ and $\Lambda$.

(2)\ Notice that the following asymptotic behavior of modified Bessel function \cite{Abramowitz1964}:
\begin{eqnarray}\label{2.16}
\begin{cases}
I_0(r)=1+\frac{\frac{1}{4}r^2}{(1!)^2}+o(r^2),\\
I_1(r)=\frac{1}{2}r+\frac{\frac{1}{4}r^3}{(2!)^2}+o(r^3),\\
K_0(r)=-\{\ln(\frac{1}{2}r)+\gamma\}I_0(r)+\frac{\frac{1}{4}r^2}{(1!)^2}+o(r^2),\\
K_1(r)=\frac{1}{r}I_0(r)+\{\ln(\frac{1}{2}r)+\gamma\}I_1(r)+o(r),\\
\end{cases}
\end{eqnarray}
where $\gamma$ is the Euler constant. A combination of (\ref{2.16}) and the definitions of operator norm for the DN map yields (\ref{2.151}).

\end{proof}

The following lemma was given in \cite{Segura2011} (page 526, (72)).
\begin{lem}\label{lem1}
For $n\geq0$ and $r>0$ the following holds:
\begin{equation}\label{2.17}
r\frac{I_n'(r)}{I_n(r)}>\sqrt{(n+1)^2+r^2}-1
\end{equation}
\end{lem}

\begin{lem}\label{lem2}
For $\eta>0$ and $r>0$ the function:
\begin{equation}\label{2.18}
F(\eta)=\eta\frac{I_1(\eta^{-1}r)}{I_0(\eta^{-1}r)}
\end{equation}
is strictly monotone increasing respect to $\eta$.
\end{lem}

\begin{proof}
By direct calculation, we derive that
\begin{equation}\label{2.19}
F'(\eta)=\eta^{-1}r\left(\frac{I_1(\eta^{-1}r)}{I_0(\eta^{-1}r)}-\theta_+\right)\left(\frac{I_1(\eta^{-1}r)}{I_0(\eta^{-1}r)}-\theta_-\right),
\end{equation}
where
\begin{align*}
\theta_+=(\eta^{-1}r)^{-1}\left(-1+\sqrt{1+(\eta^{-1}r)^2}\right),\\
\theta_-=(\eta^{-1}r)^{-1}\left(-1-\sqrt{1+(\eta^{-1}r)^2}\right).
\end{align*}
From Lemma \ref{lem1}, it implies $\frac{I_1(\eta^{-1}r)}{I_0(\eta^{-1}r)}>\theta_+$. Moreover, notice the positive of $I_n(r)$ for $r>0$, it deduces
$\frac{I_1(\eta^{-1}r)}{I_0(\eta^{-1}r)}-\theta_->0$. Then $F'(\eta)>0$ and $F'(\eta)$ is strictly increasing as asserted.
\end{proof}

In order to get the uniqueness and nonuniqueness for piecewise constant potential reconstruction, we introduce the notations from \cite{Jaeger1941}:
\begin{align}\label{2.23}
&D(x,y)=I_0(x)K_0(y)-K_0(x)I_0(y),\\
&D_{r,s}(x,y)=\frac{\partial^{r+s}}{\partial x^r\partial y^s}D(x,y).
\end{align}

Then the DN map can be rewritten as
\begin{equation}\label{2.27}
\Lambda_{\sigma_1,r_1}(f)=\frac{I_0\left(\frac{r_1}{\sqrt{\sigma_1}}\right)D_{1,1}(1,r_1)-\sqrt{\sigma_1}I_1\left(\frac{r_1}
{\sqrt{\sigma_1}}\right)D_{1,0}(1,r_1)}{I_0\left(\frac{r_1}{\sqrt{\sigma_1}}\right)D_{0,1}(1,r_1)-\sqrt{\sigma_1}I_1\left(\frac{r_1}
{\sqrt{\sigma_1}}\right)D(1,r_1)}f.
\end{equation}

The following properties are trivial, and also from \cite{Jaeger1941}.
\begin{align}
&D(x,y)D_{1,0}(x,z)-D(x,z)D_{1,0}(x,y)=x^{-1}D(z,y),\label{2.28}\\
&D(x,y)D_{1,1}(x,z)-D_{0,1}(x,z)D_{1,0}(x,y)=x^{-1}D_{1,0}(z,y),\\
&D_{1,1}(x,y)D_{0,1}(x,z)-D_{0,1}(x,y)D_{1,1}(x,z)=-x^{-1}D_{1,1}(z,y),\\
&D_{0,1}(x,y)=-D_{1,0}(y,x),\\
&D_{1,0}(x,x)=x^{-1}\label{2.32}
\end{align}

Then, the uniqueness theorem can be expressed as follows
\begin{thm}\label{th2}
(Uniqueness) For arbitrary fixed $r_1\in(0,1)$, and any $\sigma_j>0$, $j=1,2$, assume that $\Lambda_{\sigma_1,r_1}(f)=\Lambda_{\sigma_2,r_1}(f)$. Then,\\
\begin{equation}\label{2.33}
\sigma_1=\sigma_2.
\end{equation}
\end{thm}

\begin{proof}
Since $\Lambda_{\sigma_1,r_1}(f)=\Lambda_{\sigma_2,r_1}(f)$, by expression formula (\ref{2.27}), we find
\begin{align*}
&\frac{I_0\left(\frac{r_1}{\sqrt{\sigma_1}}\right)D_{1,1}(1,r_1)-\sqrt{\sigma_1}I_1\left(\frac{r_1}
{\sqrt{\sigma_1}}\right)D_{1,0}(1,r_1)}{I_0\left(\frac{r_1}{\sqrt{\sigma_1}}\right)D_{0,1}(1,r_1)-\sqrt{\sigma_1}I_1\left(\frac{r_1}
{\sqrt{\sigma_1}}\right)D(1,r_1)}\\
=&\frac{I_0\left(\frac{r_1}{\sqrt{\sigma_2}}\right)D_{1,1}(1,r_1)-\sqrt{\sigma_2}I_1\left(\frac{r_1}
{\sqrt{\sigma_2}}\right)D_{1,0}(1,r_1)}{I_0\left(\frac{r_1}{\sqrt{\sigma_2}}\right)D_{0,1}(1,r_1)-\sqrt{\sigma_2}I_1\left(\frac{r_1}
{\sqrt{\sigma_2}}\right)D(1,r_1)}.
\end{align*}
From (\ref{2.28})-(\ref{2.32}), then by straightforward calculation, we derive that
\begin{equation*}\label{}
\sqrt{\sigma_1}\frac{I_1(\frac{r_1}{\sqrt{\sigma_1}})}{I_0(\frac{r_1}{\sqrt{\sigma_1}})}
=\sqrt{\sigma_2}\frac{I_1(\frac{r_1}{\sqrt{\sigma_2}})}{I_0(\frac{r_1}{\sqrt{\sigma_2}})},
\end{equation*}
and therefore, by using monotonicity Lemma \ref{lem2}, it deduces $\sigma_1=\sigma_2$.
\end{proof}

\section{Nonuniqueness for the potential reconstruction in core-shell structure}
In this section, we prove the nonuniqueness of potential reconstruction problem, when the radius $r_1$ and potential coefficient $\sigma_1$ satisfy some conditions.
\begin{thm}\label{th3}
(Nonuniqueness) Assume that $r_j\in(0,1)$, $\sigma_j>0$, $j=1,2$, furthermore $\{r_1, \sigma_1\}$ and $\{r_2, \sigma_2\}$ satisfy
\begin{equation}\label{2.34}
D(r_1, \sigma_1, r_2, \sigma_2)=0.
\end{equation}
Then,
\begin{equation}\label{2.35}
\Lambda_{\sigma_1,r_1}=\Lambda_{\sigma_2,r_2},
\end{equation}
where
\begin{equation}\label{cy}
D(r_1, \sigma_1, r_2, \sigma_2)=
\left|\begin{array}{cccc}
D_{1,0}(r_1,r_2)&\sqrt{\sigma_1}I_1\left(\frac{r_1}
{\sqrt{\sigma_1}}\right)&D_{1,1}(r_1,r_2)\\
I_0\left(\frac{r_2}{\sqrt{\sigma_2}}\right)&0&\sqrt{\sigma_2}I_1\left(\frac{r_2}
{\sqrt{\sigma_2}}\right)\\
D(r_1,r_2)&I_0\left(\frac{r_1}{\sqrt{\sigma_1}}\right)&D_{0,1}(r_1,r_2)
\end{array}\right|.
\end{equation}
\end{thm}

\begin{proof}
Similar to the proof of Theorem \ref{th2}. the condition (\ref{2.34}) is equivalent to
\begin{align*}
&\frac{I_0\left(\frac{r_1}{\sqrt{\sigma_1}}\right)D_{1,1}(1,r_1)-\sqrt{\sigma_1}I_1\left(\frac{r_1}
{\sqrt{\sigma_1}}\right)D_{1,0}(1,r_1)}{I_0\left(\frac{r_1}{\sqrt{\sigma_1}}\right)D_{0,1}(1,r_1)-\sqrt{\sigma_1}I_1\left(\frac{r_1}
{\sqrt{\sigma_1}}\right)D(1,r_1)}\\
=&\frac{I_0\left(\frac{r_2}{\sqrt{\sigma_2}}\right)D_{1,1}(1,r_2)-\sqrt{\sigma_2}I_1\left(\frac{r_2}
{\sqrt{\sigma_2}}\right)D_{1,0}(1,r_2)}{I_0\left(\frac{r_2}{\sqrt{\sigma_2}}\right)D_{0,1}(1,r_2)-\sqrt{\sigma_2}I_1\left(\frac{r_2}
{\sqrt{\sigma_2}}\right)D(1,r_2)}.
\end{align*}
Here we have used the properties (\ref{2.28})-(\ref{2.32}). Then, by the definition of DN map, it follows that
\begin{equation*}
\Lambda_{\sigma_1,r_1}=\Lambda_{\sigma_2,r_2}.
\end{equation*}

\end{proof}

\begin{rem}
In fact, (\ref{2.34}) is the sufficient and necessary condition for equation $\Lambda_{\sigma_1,r_1}=\Lambda_{\sigma_2,r_2}$.
\end{rem}

\section{Numerical experiments and discussions}
\subsection{Potential reconstruction by using Tikhonov regularization}
\label{sec:Numerical-results}
In this subsection, we apply Tikhonov regularization method to recover the potential, then present some numerical examples to illustrate the effectiveness of algorithm and to verify the nonuniqueness result.

By setting
\[
f=1  \quad \text{and}  \quad \psi|_{r=0} \quad \text{is a bounded value}
\]
in system \eqref{2.6}. We choose the step $\Delta r=0.0001$, and use the finite difference method to solve the forward problem in the core-shell structure. Then the Neumann data $g$ can be obtained by using the following forward difference scheme
\[
g=\frac{\partial\psi}{\partial r}\bigg|_{r=1}\approx\frac{\tilde{\psi}(1)-\tilde{\psi}(1-\Delta r)}{\Delta r},
\]
where $\tilde{\psi}$ is the finite difference solution of system \eqref{2.6}.

The noisy data are generated by
\begin{equation}\label{5.36}
g^{\delta}=g+\delta\zeta,
\end{equation}
here $\zeta$ is is a Gaussian random variable with zero mean and unit standard deviation and $\delta$ indicates
the noise level.

By using the noisy data $g^{\delta}$ generated from (\ref{5.36}), we introduce the following Tikhonov regularization functional \cite{CR.Vogel-2002}
\begin{equation}\label{regular-functional}
 T_\alpha(\sigma_1)= \frac{1}{2}|\Lambda_{\sigma_1}(f)-g^{\delta}|^2+\frac{1}{2}\alpha\sigma_1^2,
\end{equation}
where $\alpha$ is a regularization parameter, which will be selected by the discrepancy principle \cite{CR.Vogel-2002} in our numerical simulation. The regularization solution is defined as the minimizer of functional (\ref{regular-functional}), i.e., regularization solution
\begin{equation}\label{5.38}
\sigma_{1}^{\alpha,\delta}=\arg\min_{\sigma_1>0} T_\alpha(\sigma_1).
\end{equation}
We will use the Newton's algorithm to find the minimizer (regularization solution). To show the accuracy of numerical solutions, we compute the absolute error denoted by
\[
\varepsilon_{abs}=|\sigma_{1}^{\alpha,\delta}-\sigma_{1}|,
\]
here $\sigma_{1}$ is the exact potential coefficient.



\begin{exm}\label{exm:1}
Consider system \eqref{2.6} with $r_{1}=0.7$, $\sigma_{1}=0.9$, $\psi|_{r=0}=\frac{692}{887}$.


\begin{table}[!htbp]
\caption{Numerical results for Example 5.1 with different noise level $\delta$.}\label{table 0.7}
  \centering
  \begin{tabular}{cccc}
  \hline
  $\delta$&$\alpha$&$\sigma_1^{\alpha,\delta}$&$\varepsilon_{abs}$\\\hline
  0.1 & 1.36e-03 & 0.9007  & 7.4295e-04\\\hline
  0.01 & 1.29e-04 & 0.9003 & 3.3653e-04
  \\\hline
  0.001 & 5.4e-06 & 0.9002 & 2.4817e-04
  \\\hline
  \end{tabular}
\end{table}
\end{exm}

\begin{exm}\label{exm:2}
Let $r_{1}=0.8$, $\sigma_{1}=1.5$, $\psi|_{r=0}=\frac{199}{240}$ in system \eqref{2.6}

\begin{table}[!htbp]
\caption{Numerical results for Example 5.2 with different noise level $\delta$.}\label{table 0.8}
  \centering
  \begin{tabular}{cccc}
  \hline
  $\delta$&$\alpha$&$\sigma_1^{\alpha,\delta}$&$\varepsilon_a$\\\hline
  0.1 & 5.56e-04 & 1.4999  & 1.3517e-04\\\hline
  0.01 & 5.19e-05 & 1.5003 & 2.9780e-04
  \\\hline
  0.001 & 2.1e-06 & 1.5000 & 4.6362e-05
  \\\hline
  \end{tabular}
\end{table}
\end{exm}

%

From Table 1 and Table 2, it shows that the regularization solution $\sigma_1^{\alpha,\delta}$ is in excellent agreement with
the exact solution $\sigma_1$, even though the noisy level $\delta$ reach 10\%. The numerical results also confirm that one measurement data is enough to determine the unknown potential $\sigma_1$ uniquely.


\subsection{Verifying the nonuniqueness result}
In this subsection we will verify the nonuniqueness result (Theorem 4.1). In fact, we first fixed $r_{1}$, $r_{2}$, $\sigma_{1}$, and then find $\sigma_{2}$ from equation (\ref{2.34}). According to Theorem 4.1, the corresponding DN map must be same. Notice that the DN map $\Lambda_{\sigma_1,r_1}:\ \mathbb{C}\rightarrow\mathbb{C}$ is a linear multiplier operator. Hence, we only need to compare $\Lambda_{\sigma_1,r_1}(f)$ with $\Lambda_{\sigma_2,r_2}(f)$ with $f=1$. Furthermore, we introduce absolute error $\varepsilon_{abs}(\Lambda)=|\Lambda_{\sigma_1,r_1}(1)-\Lambda_{\sigma_2,r_2}(1)|$ to show the difference between $\Lambda_{\sigma_1,r_1}$ and $\Lambda_{\sigma_2,r_2}$. In the following, we still use the finite difference method to calculate $\Lambda_{\sigma_1,r_1}(1)$ and $\Lambda_{\sigma_2,r_2}(1)$, and the difference step $\Delta r=\frac{1}{N}$.

\begin{exm}\label{exm:4}
Fixed $r_{1}=0.3$, $r_{2}=0.7$, $\sigma_{1}=2$, $\psi|_{r=0}=\frac{852}{1067}$, by solving equation (\ref{2.34}), we find $\sigma_{2}\approx1.0161$.

\begin{table}[h]
\caption{Numerical results for Example 5.3 with errors versus different $N$.}\label{main1}
  \centering
  \begin{tabular}{cccc}
  \hline
  $N$&$\Lambda_{\sigma_1,r_1}(1)$&$\Lambda_{\sigma_2,r_2}(1)$&$\varepsilon_{abs}(\Lambda)$\\\hline
  100 & 0.4431 & 0.4409 & 2.2122e-03\\\hline
  200 & 0.4448 & 0.4438 & 1.0481e-03
  \\\hline
  400 & 0.4457 & 0.4452 & 5.1991e-04
  \\\hline
  800 & 0.4461 & 0.4459 & 2.5845e-04
  \\\hline
  \end{tabular}
\end{table}


\end{exm}

\begin{exm}\label{exm:5}
Fixed $r_{1}=0.8$, $r_{2}=0.4$, $\sigma_{1}=0.5$, $\psi|_{r=0}=\frac{1139}{1658}$, we find $\sigma_{2}\approx0.0373$.

\begin{table}[!htbp]
\caption{Numerical results for Example 5.4 with errors versus different $N$.}\label{main2}
  \centering
  \begin{tabular}{cccc}
  \hline
  $N$&$\Lambda_{\sigma_1,r_1}(1)$&$\Lambda_{\sigma_2,r_2}(1)$&$\varepsilon_{abs}(\Lambda)$\\\hline
  100 & 0.4234 & 0.4254 & 2.0125e-03\\\hline
  200 & 0.4267 & 0.4278 & 1.0036e-03
  \\\hline
  400 & 0.4284 & 0.4289 & 5.0138e-04
  \\\hline
  800 & 0.4293 & 0.4295 & 2.5071e-04
  \\\hline
  \end{tabular}
\end{table}


\end{exm}

\section{Conclusions}
In this article, we considered a potential reconstruction problem in quantum fields from one boundary measurement data. We proved the corresponding uniqueness theorem and non-uniqueness result. Based on the Tikhonov regularization scheme, some numerical examples are given to confirm the theoretical analysis.

\section*{Acknowledgments}
The work described in this paper was supported by the NSF of China (11301168).

\end{document}